\documentclass[a4paper,11pt]{article}
\usepackage{fullpage}
\usepackage{amsfonts} 
\usepackage{amssymb} 
\usepackage{latexsym} 
\usepackage{amsmath}
\usepackage{amsthm}
\usepackage[all]{xy}
\usepackage{mathrsfs}
\newcommand{\nn}{\mathbb{N}}
\newcommand{\zz}{\mathbb{Z}}
\newcommand{\qq}{\mathbb{Q}}
\newcommand{\cc}{\mathbb{C}} 
\newcommand{\pp}{\mathbb{P}} 
\newcommand{\lra}{\longrightarrow}                         
\newcommand{\ra}{\rightarrow}

\newcommand{\xra}{\xrightarrow}

\newcommand{\gen}[1]{\left<#1\right>}%
\newcommand{\set}[1]{\left\{ #1 \right\}}%
\newcommand{\abs}[1]{\left\vert #1 \right\vert}%
\renewcommand{\aa}{\mathbb{A}}%
\newcommand{\orb}{\text{orb}}%

\newcommand{\mbf}{\mathbf}
\newcommand{\bs}{\boldsymbol}

\newtheorem{thm}{Theorem}[section]%
\newtheorem{lem}[thm]{Lemma}%

\newtheorem{prop}[thm]{Proposition}%
\theoremstyle{definition} 
\newtheorem{rem}[thm]{Remark}%
\newtheorem{overview}[thm]{Overview}
\newtheorem{convention}[thm]{Convention}
\newtheorem*{ack}{Acknowledments}

\DeclareMathOperator{\Hom}{Hom}%
\DeclareMathOperator{\lcm}{lcm}%

\DeclareMathOperator{\spec}{Spec}
\DeclareMathOperator{\im}{Im}

\title{Mirror fibrations and root stacks of weighted projective spaces}%
\author{Ignacio de Gregorio and \'Etienne Mann}%
\date{18 October, 2007}

\begin{document}

\maketitle

\abstract{We show that the orbifold Chow ring of a root stack over a well--formed weighted
  projective space can be naturally seen as the Jacobian algebra of a function on a
  singular variety. }

\section{Introduction}
\label{sec:introduction}

According to A.~Givental (\cite{Gmttci98}) and S.~Barannikov (\cite{Bms}), the mirror
partner of the projective space $\pp^n$ is the function $f_1=x_0+\dots+x_n$ on the torus
defined by $x_0\cdots x _n=1$. This mirror theorem states an isomorphism between the
Frobenius manifolds obtained by unfolding $f_{1}$ and the quantum cohomology of $\pp^n$.
To explain the motivation behind this article let us think of $f=x_0+\dots+x_n$ as defined
on the fibration $\pi:\aa^{n+1}\ra\aa^1$ given by $\pi(x_0,\dots,x_n)=x_0\cdots x_n$.
From this point of view, it is natural to consider deformations of $f$ as unfoldings
$F(x,t):\aa^{n+1}\times\aa^k\ra\cc$ satisfying $F(x,0)=f(x)$, together with the
equivalence relation induced by commutative diagrams
\begin{equation}
  \label{eq:1}
  \begin{split}
    \xymatrix{\aa^{n+1}\times\aa^k\ar[rr]^{\phi}\ar[dd]_{\pi\times\text{id}_{\aa^k}}\ar[rd]^{F}&&
      \aa^{n+1}\times\aa^{k'}\ar[dd]^{\pi\times\text{id}_{\aa^{k'}}}\ar[ld]_{F'}\\
      &\cc\\
      \aa^1\times\aa^k\ar[rr]_{\psi}&&\aa^1\times\aa^{k'}}
  \end{split}
\end{equation}
Standard techniques (e.g.~\cite{IDGdd-06}) show that, at least at the level of germs, the
tangent space to the corresponding deformation functor, denoted by $T^1_{f/\pi}$, is given
by the algebra
\begin{equation}
  \label{eq:2}
  T^1_{f/\pi}=\frac{\cc[x_0,\dots,x_n]}{(\pi)+\Theta_\pi(f)}
\end{equation}
where $\Theta_\pi$ denotes the vector fields on $\aa^{n+1}$ tangent to all the fibres of
$\pi$. We will refer to this algebra as the Jacobian algebra of $f$ at the $0$-fibre of
$\pi$.  In the case of the mirror of $\pp^n$, it is readily seen that $\Theta_\pi$ is
freely generated by the vector fields
\begin{equation}
  \label{eq:3}
  x_i\partial_{x_i}-x_{i+1}\partial_{x_{i+1}}, i=0,\dots,n-1.
\end{equation}
Therefore we have
\begin{equation}
  \label{eq:4}
  T^1_{f/\pi}=\frac{\cc[x]}{(x^{n+1})}\simeq H^*(\pp^n;\cc)
\end{equation}
It is in this sense that seems natural to us to call the pair $(f,\pi)$ the mirror
fibration of $\pp^n$.

Let $p_{0}, \ldots ,p_{n}$ be integer greater or equal than one. In the case of the
weighted projective space $\pp(p_0,\dots,p_n)$, it has also been recently proved that the
restriction of the function $f=x_0+\dots+x_n$ to the torus $x_0^{p_0}\dots x_n^{p_n}=1$ is
the mirror partner of $\pp(p_0,\dots,p_n)$. This result appears as a conjecture in
\cite{mann-2006} and it follows after the calculation of the small quantum orbifold
cohomology of $\pp(p_0,\dots,p_n)$ by T.~Coates {\em et al.}  (\cite{CCLTqocwps}).

In this note, we construct a mirror fibration, in the sense explained above, for a toric
orbifold whose coarse moduli space is a well--formed weighted projective space.  In order
to state our main result, we first introduce some notations.

A sequence of weights $\mbf{p}:=(p_0,\dots,p_n) \in (\nn_{>0})^{n+1}$ is called
\textit{well--formed} if for any $i\in\{0, \ldots ,n\}$ we have $\gcd(p_{0}, \ldots
,\widehat{p_{i}}, \ldots ,p_{n})=1$.  A weighted projective space $\pp(\mbf{p})$ is called
well--formed if its weights are well--formed.

As explained in Section 5 of \cite{FMNts-07}, a toric orbifold whose coarse moduli space
is a well--formed weighted projective space $\pp(\mbf{p})$ can be encoded by a
$(n+1)$--tuple $\mbf{w}:=(w_{0}, \ldots ,w_{n}) \in (\nn_{>0})^{n+1}$ which are the
multiplicities of the toric divisors. Such a toric orbifold is denoted by
$\mathcal{X}(\mbf{w},\mbf{p})$.

Now we state our main theorem.

\begin{thm}\label{thm:intro}
  Let $\mbf{p}:=(p_0,\dots,p_n) \in (\nn_{>0})^{n+1}$ be a sequence of well--formed
  weights. Let $\mbf{w}:=(w_0,\dots,w_n) \in (\nn_{>0})^{n+1}$.  There exists a fibration
  $\pi_{\mbf{p}}:\mathcal{Y}(\mbf{p})\ra\mathcal{C}(\mbf{p})$ over a rational curve
  together with a function $f_{\mbf{w}}:\mathcal{Y}(\mbf{p}) \ra\cc$ such that
  \begin{enumerate}
  \item the generic fibre $\pi_{\mbf{p}}^{-1}(t), t\not=0$ is isomorphic to the torus
    $x_0^{p_0}\dots x_n^{p_n}=1$ and $f_{\mbf{w}}$ is given by
    $x_0^{w_{0}}+\dots+x_n^{w_{n}}$ under this isomorphism;
  \item we have a ring isomorphism
    \begin{equation}
      \label{eq:5}
      T^1_{f_{\mbf{w}}/\pi_{\mbf{p}}}\simeq A^*_{\orb}(\mathcal{X}(\mbf{w},\mbf{p});\cc)
    \end{equation}
    where the right-hand side denotes the orbifold Chow ring of
    $\mathcal{X}(\mbf{w},\mbf{p})$.
  \end{enumerate}
\end{thm}

The definition of the orbifold cohomology can be found in \cite{CRnco} or in
\cite{AGVaoqc}.  To prove the theorem above, we only use the fact that the
$\gcd{\mbf{p}}=1$ and not that the $p_{i}$'s are well--formed.  Nevertheless, it is not a
more general case as it is explain in Proposition \ref{prop:iso,mirror,fibration}. We put
this assumption in the theorem for the following reason.  As stated, this theorem
highlights the role of the fibration and the function with respect to the toric orbifold
$\mathcal{X}(\mbf{w},\mbf{p})$.  Indeed, the coarse moduli space is encoded by the
fibration $\pi_{\mbf{p}}$ whereas the root of toric divisors are encoded by the function
$f_{\mbf{w}}$.  It would be interesting to see how to encode the essentially gerbe
structure appeared in Section 6 of \cite{FMNts-07} for the mirror fibration.

The reader wanting to know what $\pi_{\mbf{p}}$ and $f_{\mbf{w}}$ look like, might wish to
have a look at the last section of this note before reading any further.

\begin{convention}
  An \textit{orbifold} is a smooth DM stack of finite type over $\cc$ with trivial generic
  stabiliser.
\end{convention}

\section{Orbifold Chow ring of smooth toric DM stacks}
\label{sec:toric-stacks-polyt}

First we recall some general facts on smooth toric DM stacks.  According to
\cite{BCSocdms05} a stacky fan denoted by $\bs{\Sigma}$ is a triple $(N,\Sigma,\beta)$
where $N$ is a finitely generated abelian group $N$, $\Sigma$ is a simplicial fan in
$N_\qq:=N\otimes_\zz\qq$ with $n+1$ rays and $\beta:\zz^{n+1}\ra N$ is a group
homomorphism such that the image of the standard basis, denoted by $(\mbf{e}_{0}, \ldots
,\mbf{e}_{n})$ of $\zz^{n+1}$ generates the rays of $\Sigma$.  To this combinatorial data,
one can associate a smooth DM stack denoted by $\mathcal{X}(\bs{\Sigma})$. We will not use
explicitly this construction so we refer to \cite{BCSocdms05} for it.

Denote by $\qq[N]^\Sigma$ denotes the deformed group ring, that is, the underlying vector
space is simply $\qq[N]$ but the multiplication has been deformed according to the rule
\begin{equation}
  \label{eq:10} {y}^{c_1}\cdot
  {y}^{c_2} =
  \begin{cases} {y}^{c_{1}+c_{1}} & \text{if
      there exists a cone $\sigma\in\Sigma$ such that
      $c_{1},c_{2}\in\sigma$}\\ 0 & \text{otherwise.}
  \end{cases}
\end{equation}
Let $\theta\in N^{\vee}:=\Hom_{\zz}(N,\zz)$. We define the $\qq[N]^{\Sigma}$--linear
morphism
\begin{align}
  \xi_{\theta}: \qq[N]^{\Sigma} &\lra  \qq[N]^{\Sigma} \\
  y^{c} &\longmapsto \theta(c)y^{c} \nonumber
\end{align}
One can prove easily the following lemma.
\begin{lem}\label{lem:derivation}
  For any $\theta\in N^{\vee}$, the linear morphism $\xi_{\theta}$ is a derivation of
  $\qq[N]^{\Sigma}$.
\end{lem}

We finish our recall by stating the main result of \cite{BCSocdms05}.  The ring
$A_{\orb}^{\star}(\mathcal{X}(\bs{\Sigma}))$ is isomorphic to
\begin{equation}
  \label{eq:9}
  \frac{\qq[N]^{\Sigma}}
  {\langle\xi_{\theta}(\sum_{i=0}^{n} y^{\beta(\mbf{e}_{i})}):\theta\in
    N^{\vee}\rangle}.
\end{equation}

\begin{rem}
  From Lemma \ref{lem:derivation} and the description of the orbifold Chow ring above, it
  seems natural to see $\qq[N]^{\Sigma}$ as the fibration and $\sum_{i=0}^{n}
  y^{\beta(\mbf{e}_{i})}$ as a function. We will explicit this in the next section on our
  examples.
\end{rem}

Let $\mbf{p}:=(p_0,\dots,p_n) \in (\nn_{>0})^{n+1}$ be a sequence of well--formed weights.
Let $\mbf{w}:=(w_0,\dots,w_n) \in (\nn_{>0})^{n+1}$.  Now, we describe the stacky fan of
the toric orbifold $\mathcal{X}(\mbf{w},\mbf{p})$.

The finitely abelian group $N$ is $\zz^{n+1}/\gen{\sum_{i=0}^{n}p_{i}\mbf{e}_{i}}$. As the
$p_{i}$ are coprime, the abelian group $N$ is free of rank $n$.  The vector space
$N\otimes_{\zz}\qq$ comes equipped with a natural simplicial fan $\Sigma$ given by the
projections of the non-negative coordinate subspaces in $\zz^{n+1}\otimes_\zz\qq$. More
precisely, for $k\in\{0, \ldots ,n\}$, the set of $k$-dimensional cones of $\Sigma$ is
given by
\begin{equation}
  \label{eq:6} \sigma_J:=\set{\sum_{j\in
      J}\lambda_j[\mathbf{e}_j]:\lambda_j\geq 0 \in \qq}
\end{equation} where $J\subset\set{0,\dots,n}$ runs through all the
subsets with $k$ elements. In order to define the homomorphism
$\beta$, we choose a point $w_i\mathbf{e}_i\in\zz^{n+1},~w_i>0$. If
$W$ denotes the diagonal matrix $(w_0,\dots,w_n)$, we define $\beta$
as the composite
\begin{equation}
  \label{eq:7} \beta:\zz^{n+1}\xra{W}\zz^{n+1}\ra
  N.
\end{equation} 
We denote by $\mathcal{X}(\mathbf{w},\mathbf{p})$ the smooth DM stack associated to the
stacky fan $(N,\Sigma,\beta)$.

To have a more geometrical grasp on $\mathcal{X}(\mathbf{w},\mathbf{p})$, we use the
bottom--up construction and Section 7 of \cite{FMNts-07}.  We deduce that the coarse
moduli space of $\mathcal{X}(\mathbf{w},\mathbf{p})$ is
$\mathcal{X}(\mathbf{1},\mathbf{p})$ where all the $w_{i}$'s are $1$.  It is a
straightforward computation to see that $\mathcal{X}(\mathbf{1},\mathbf{p})$ is the
well--formed weighted projective space $\pp(\mbf{p})$.  Denote by
$\mathcal{T}:=[(\cc^{\ast})^{n+1}/\cc^{\ast}]$ where the action of $\cc^{\ast}$ on
$(\cc^{\ast})^{n+1}$ is given by :
\begin{equation}
  \lambda\cdot(x_{0}, \ldots ,x_{n}):=(\lambda^{p_{0}}x_{0}, \ldots ,\lambda^{p_{n}}x_{n}).
\end{equation}
The $\mathcal{X}(\mathbf{1},\mathbf{p})\setminus \mathcal{T}$ is a simple normal crossing
divisor with irreducible components denoted by $\mathcal{D}_{i}$. Denote
$\bs{\mathcal{D}}:=(\mathcal{D}_{0}, \ldots ,\mathcal{D}_{n})$. The $\mbf{w}$--th root
stack of $(\mathcal{X}(\mathbf{1},\mathbf{p}),\bs{\mathcal{D}})$ is the fibber product
\begin{equation}
  \xymatrix{
    \ar[r] \ar@{}[rd]|{\square}
    \ar[d]\sqrt[\mbf{w}]{\bs{\mathcal{D}}/\mathcal{X}(\mathbf{1},\mathbf{p})}&
    [\aa^{n+1}/(\cc^{\ast})^{n+1}]\ar[d]^{\wedge \mbf{w}}\\
    \ar[r]\mathcal{X}(\mathbf{1},\mathbf{p})& [\aa^{n+1}/(\cc^{\ast})^{n+1}]
  }
\end{equation}
where the stack morphism $\wedge \mbf{w}:[\aa^{n+1}/(\cc^{\ast})^{n+1}] \to
[\aa^{n+1}/(\cc^{\ast})^{n+1}]$ is defined by sending $x_{i}\mapsto x_{i}^{w_{i}}$ and
$\lambda_{i}\mapsto \lambda_{i}^{w_{i}}$ where $x_{i}$ (resp. $\lambda_{i}$) is the
coordinates of $\aa^{n+1}$ (resp. $(\cc^{\ast})^{n+1}$ ).  Section 7 of \cite{FMNts-07},
we deduce that $\mathcal{X}(\mathbf{w},\mathbf{p})$ is isomorphic to
$\sqrt[\mbf{w}]{\bs{\mathcal{D}}/\mathcal{X}(\mathbf{1},\mathbf{p})}$.

 \begin{rem}\label{rem:iso,w,p}
   Let $a\in \nn$ such that $\gcd(a,p_{n})=1$. Then we have that
   \begin{equation}
     \mathcal{X}((w_0, \ldots ,w_n),(ap_{0}, \ldots ,ap_{n-1},p_n)) \simeq 
     \mathcal{X}((w_0, \ldots ,aw_n),(p_{0}, \ldots ,p_{n}))
   \end{equation}
 \end{rem}

 \section{Orbifold Chow ring as Jacobian algebra}
 \label{sec:orbifold-chow-ring}

 In this section we construct the fibration with the properties described in the
 introduction.

\begin{overview}
  Looking at the orbifold Chow ring in \eqref{eq:9}, we will see $\qq[N]^{\Sigma}$ as a
  ring defining the fibration $\pi_{\mbf{p}}:\mathcal{Y}(\mbf{p})\to \mathcal{C}(\mbf{p})$
  and $\sum_{i=0}^{n} y^{\beta(\mbf{e}_{i})}$ as the function
  $f_{\mbf{w}}:\mathcal{Y}(\mbf{p}) \to \qq$. Using this idea, we will see the ring
  $A^{\ast}_{\orb}(\mathcal{X}(\mbf{w},\mbf{p}))$ as a Jacobian algebra.
\end{overview}

We first wish to express $\qq[N]^\Sigma$ as the quotient of a polynomial algebra by an
ideal. In order to do so, we define
\begin{align}
  \widetilde{\alpha} : \zz^{n+1} &\lra (\qq_{\geq 0})^{n+1} \\
  \mbf{a}:=(a_{0}, \ldots ,a_{n}) & \longmapsto \mbf{a}-\gamma(\mbf{a})\mbf{p}\nonumber
\end{align}
where $\gamma(\mathbf{a}):=\min\set{\frac{a_i}{p_i}:i=0,\dots,n}$.

The map $\widetilde{\alpha}$ admits the following interpretation:
$\widetilde{\alpha}(\mathbf{a})$ is the point of intersection of the line
$\mathbf{a}+\lambda\mathbf{p}$ with the set $\set{(x_0,\dots,x_n)\in(\qq_{\geq 0})^{n+1}
  \mbox{ such that } x_0\cdots x_n=0 }$. It thus descends to a map $\alpha:N\to (\qq_{\geq
  0})^{n+1}$. Denote also by $\mbf{a}$ the class of $\mbf{a}$ in $N$. We also see from
this interpretation that
\begin{equation}
  \label{eq:12}
  \alpha(\mathbf{a}+\mathbf{a}')=\alpha(\alpha(\mathbf{a})+\alpha(\mathbf{a}'))
\end{equation} 
and if $\sigma_J\subset N_{\qq}$ denotes the cone of $\Sigma$ defined in (\ref{eq:6}) then
\begin{equation}
  \label{eq:13} \alpha(\sigma_J) \subset
  \set{(x_0,\dots,x_n)\in(\qq_{\geq 0})^{n+1} \mbox{such that } x_{i}=0 \mbox{ for } i\not\in J}.
\end{equation} 
Notice that $\mbf{a},\mbf{a'}\in N$ are not in the same cone if and only if for any
$i\in\set{0, \ldots ,n}$, $a_{i}+a'_{i}>0$. It follows that
$\alpha(\mathbf{a}+\mathbf{a}')=\alpha(\mathbf{a})+\alpha(\mathbf{a}')$ if and only if
there exists $\sigma\in\Sigma$ with $\mathbf{a},\mathbf{a}'\in \sigma$. Denote by
$S\subset (\qq_{\geq 0})^{n+1}$ the semigroup (with unity) generated by $\alpha(N)$.
Denote by $\qq[S]$ the algebra generated by $S$.  As usual, we denote by $z_{i}$ the
element $\alpha(\mbf{e}_{i})$ in $\qq[S]$ and we write $\mbf{z^{b}}:=z_0^{b_0}\dots
z_n^{b_n}$ for the element $\mbf{b}\in \qq[S]$.  The definition of $\qq[N]^{\Sigma}$ and
the above discussion imply that
\begin{equation}
  \label{eq:14}
  \qq[N]^\Sigma\simeq\frac{\qq[S]}{(\set{\mathbf{z}^{\mathbf{b}}:
      \mathbf{b}=(b_0,\dots,b_n)\in
      S,b_i>0,i=0,\dots,n})}
\end{equation} 
Denote by $T\subset \qq$ the semigroup generated by $\gamma(S)$. We have the following
descriptions of $S$ and $T$:

\begin{lem}
  \label{prop:2}
  \begin{enumerate}
  \item The semigroup $S$ is generated by $\set{\mathbf{e}_i:i=0,\dots,n}$ and
    $S_0,\dots,S_n$ where
    \begin{equation}\label{eq:34}
      S_i:=\set{\frac{1}{p_i}(\overline{kp_{0}}^{p_{i}},\dots,\overline{kp_{n}}^{p_{i}}) \mbox{ where } \overline{kp_{n}}^{p_{i}} \mbox{ is the remainder in the division of $kp_{j}$ by $p_{i}$}}
    \end{equation}
  \item $T$ is generated by $1/\lcm(p_i,p_j), 0\leq i < j \leq n$.
  \end{enumerate}
\end{lem} \newcounter{tmp}

\begin{proof} (a). Let $S'$ be the semigroup generated by $\set{\mathbf{e}_i:i=0,\dots,n}$
  and $S_0,\dots,S_n$. We want to show that $S=S'$. Notice that
  $\widetilde{\alpha}(\zz^{n+1})=\alpha(N)$ and by definition generates $S$. Let
  $\mathbf{a}=(a_0,\dots,a_n)\in\zz^{n+1}$. We can assume without loss of generality that
  $a_0/p_0=\min\set{a_i/p_i: i=0,\dots,n}$. Let $u\in\nn$ such that $k:=up_{0}-a_{0} \in
  \nn$. Writing the Euclidean division $kp_i=q_ip_0+\overline{kp_{i}}^{p_{0}}$, we see
  that $a_i-(up_{i}+q_i)\geq 0$. Then
  \begin{equation}
    \label{eq:16}
    \widetilde{\alpha}(\mathbf{a})=\mathbf{a}-\frac{a_0}{p_0}\mathbf{p}
    =(0,a_1-(up_{1}+q_1),\dots,a_n-(up_{n}+q_n))+\frac{1}{p_0}(0,\overline{kp_{1}}^{p_{0}},\dots,\overline{kp_{n}}^{p_{0}})
  \end{equation} from which it follows that $S\subset S'$. For the
  reverse inclusion, it is enough to show that $S_i\subset S$. Again we
  show that $S_0\subset S$. Let
  $\frac{1}{p_{0}}(0,\overline{kp_{1}}^{p_{0}},\dots,\overline{kp_{n}}^{p_{0}})$
  be an element in $S_{0}$.
  There exists $q_{i}\in \nn$ such that for $i\in\{0, \ldots ,n\}$ we have
  \begin{equation}
    -kp_{i}=-q_{i}p_{0}+\overline{-kp_{i}}^{p_{0}}=(1-q_{i})p_{0}- \overline{kp_{i}}^{p_{0}}. 
  \end{equation}
  Dividing by $p_{0}p_{i}$, we see that ${-k}/p_{0}\leq (1-q_{i})/p_{i}$.  We deduce that
  \begin{equation}
    \alpha(-k,1-q_{1},\dots,1-q_{n})=\frac{1}{p_0}(0,\overline{kp_{1}}^{p_{0}},\dots,\overline{kp_{n}}^{p_{0}}).
  \end{equation}
  \newline


  \vspace{\parsep} (b). As before we consider the semigroup $T'$ generated by
  $1/\lcm(p_i,p_j)$ for $0\leq i<j\leq n$. In view of part (a), an element $\mathbf{s}\in
  S$ can be written as a finite sum
  \begin{equation}
    \label{eq:17} \mathbf{s}=\sum_{i=0}^n s_i\mathbf{e}_i +
    \sum_{i=0}^n \mathbf{b}_i,~ s_{i}\geq 0
  \end{equation} with $\mathbf{b}_i$ in the semigroup generated by
  $S_i$. If we write
  $\mathbf{b}_i=\frac{1}{p_i}(b_{i,0},\dots,b_{i,n})$,
  $\gamma(\mathbf{s})$ is given by
  \begin{equation}
    \label{eq:18}
    \gamma(\mathbf{s})=\min\set{\frac{1}{p_j}\left(s_j+\sum_{i=0}^n
        \frac{b_{i,j}}{p_i}\right):j=0,\dots,n}
  \end{equation} On the other hand, for any $0\leq i\leq n$ there exists $k_{i}\in \nn$ such that for any $j\in\set{0, \ldots ,n}$
  we have $b_{i,j}\equiv k_{i}p_j\!\!\mod p_i$. In particular,
  $b_{i,i}=0$ and $b_{i,j}\in\gcd(p_i,p_j)\zz$ if $i\neq j$.  As
  $\lcm(p_i,p_j)\gcd(p_i,p_j)=p_ip_j$, it follows that
  $\gamma(\mathbf{s})\in T'$. To see that $1/\lcm(p_i,p_j)\in T$ for
  $i\not=j$, choose a positive integer $\ell$ such that
  $\ell p_j\equiv \gcd(p_i,p_j)\!\!\mod p_i$. For $m\in\set{0, \ldots ,n}$, set $\ell_m\equiv kp_m\!\!\mod p_i$ with
  $0\leq \ell_m<p_m$. Notice that $\ell_{j}=\gcd(p_{i}p_{j})$. The element
  \begin{equation}
    \label{eq:19} \mathbf{b}=\frac{1}{p_i}(\ell_0,\dots,\ell_n) +
    (k_0,\dots,\stackrel{j\char41}{0},\dots,k_n)
  \end{equation} for $k_m$ sufficiently large, satisfies
  $\gamma(\mathbf{b})=1/\lcm(p_i,p_j)$.
\end{proof}

We are now ready to construct the fibration described in the introduction. Let
$\ell=1/\lcm(p_0,\dots,p_n)$ and $\overline{T}$ be the additive subgroup of $\qq$
generated by $\ell$. Let us also set $\ell_i=p_i\ell$ and denote by $\overline{S}$ the
subgroup of $\qq^{n+1}$ generated by $\ell_i\mathbf{e}_i$. We have a well-defined
commutative diagram respecting the addition:
\begin{equation}
  \label{eq:20}
  \begin{array}{cc}
    \parbox{3cm}{$\xymatrix{T\oplus S\ar[r]^{\phi^*} & \overline{S} \\
        T\ar[u]^{\pi_{\mbf{p}}^*}\ar[r]_{\psi^*} & \overline{T}\ar[u]_{\rho^*}}$} &
    \parbox{3cm}{$\phi^*(\gamma,\mathbf{b})=\mathbf{b}-\gamma\mathbf{p}$\\
      $\pi_{\mbf{p}}^*(\gamma)=(\gamma,0)$\\ $\rho^*(\lambda)=\lambda\mathbf{p}$\\
      $\psi^*(\gamma)=-\gamma$}
  \end{array}
\end{equation}
Notice that $\phi^*(\gamma,\mathbf{b})=0$ if and only if
$(\gamma,\mathbf{b})=(\gamma(\mathbf{b}),\mathbf{b}-\alpha(\mathbf{b}))$.  We denote by
$t^\gamma\mathbf{z}^{\mathbf{b}}=t^\gamma z_0^{b_0}\cdots z_n^{b_n}$ the corresponding
element in $\qq[T\oplus S]$. Consider the ideal $I\subset\qq[T\oplus S]$ generated by
\begin{equation}
  \label{eq:21}
  \set{\mathbf{z}^{\mathbf{b}}-t^{\gamma(\mathbf{b})}\mathbf{z}^{\alpha(\mathbf{b})}:
    \mathbf{b}\in S}.
\end{equation} We obtain a commutative diagram of ring homomorphisms:
\begin{equation}
  \label{eq:22}
  \begin{split} \xymatrix{\qq[T\oplus S]/I\ar[r]^{~~~~\phi^*} &
      \qq[\overline{S}] \\ \qq[T]\ar[u]^{\pi_{\mbf{p}}^*}\ar[r]_{\psi^*} &
      \qq[\overline{T}]\ar[u]_{\rho^*}}
  \end{split}
\end{equation}
Consider now the elements $f_{\mbf{w}}:=\sum_{i=0}^n
\mathbf{z}^{\beta(\mathbf{e}_i)}=z_0^{w_0}+\dots+z_n^{w_n}\in\qq[T\oplus S]$. Denote by
$x_{i}$ the element $\mbf{e}_{i}$ in $\qq[\overline{S}]$. We write
$\mathbf{x}^{\mathbf{b}}=x_0^{b_0}\cdots x_n^{b_n}$ for the element $\mbf{b}\in
\qq[\overline{S}]$. Put $\overline{f}_{\mbf{w}}=\sum_{i=0}^n
\mathbf{x}^{\beta(\mathbf{e}_i)}=\sum_{i=0}^{n}x_{i}^{w_{i}} \in \qq[\overline{S}]$.
Then, taking $\spec$ of the diagram (\ref{eq:22}) we obtain the following.
\begin{thm}
  \label{prop:3} The commutative diagram
  \begin{equation}
    \label{eq:23}
    \begin{split}
      \xymatrix{\spec\qq[\overline{S}]=:\mathbb{T}^{n+1}\ar[dd]_{\rho}\ar[rr]^{\phi}\ar[rd]^{\overline{f}_{\mbf{w}}}
        && \mathcal{Y}(\mbf{p}):=\spec\qq[T\oplus S]/I\ar[dd]^{\pi_{\mbf{p}}}\ar[ld]_{f_{\mbf{w}}} \\ & \qq \\
        \spec\qq[\overline{T}]=:\mathbb{T}\ar[rr]_{\psi} &&
        \mathcal{C}(\mbf{p}):=\spec(\qq[T])}
    \end{split}
  \end{equation} satisfies:
  \begin{enumerate}
  \item $\phi$ and $\psi$ are isomorphism over their images;
  \item $\pi_{\mbf{p}}$ is flat and
  \item the Jacobian algebra of $f$ over the $0$-fibre of $\pi_{\mbf{p}}$ is isomorphic to
    the orbifold Chow ring $A_{\orb}^*(\mathcal{X}(\mathbf{w},\mathbf{p}))$.
  \end{enumerate}
\end{thm}

\begin{proof}
  The statements (a) and (b) are clear. For (c) we notice that
  $I+(\set{t^{\gamma}:\gamma>0\in T})$ is canonically isomorphic to the right hand side of
  (\ref{eq:14}) and hence isomorphic to $\qq[N]^\Sigma$. It remains to identify the module
  $\Theta_{\pi_{\mbf{p}}}$ of $\qq[T]$-linear derivations of $\qq[T\oplus S]$ with the
  denominator of (\ref{eq:9}). According to Lemma \ref{lem:derivation}, for any $\theta\in
  N^\vee$, the application
  \begin{align}
    \label{eq:24} \xi_\theta :\qq[T\oplus S]&\lra \qq[T\oplus S] \\
    t^\gamma\mathbf{z}^{\mathbf{b}}�&\longmapsto
    \tilde{\theta}(\mathbf{b})t^\gamma\mathbf{z}^{\mathbf{b}} \nonumber
  \end{align}
  is a derivation.

  On the other hand $\xi_\theta(I)\subset I$ for $\tilde{\theta}(\qq\cdot\mathbf{p})=0$.
  Hence $\xi_\theta$ is a derivation of $\qq[T\oplus S]/I$ which is $\qq[T]$-linear by
  definition. We therefore obtain a map $\xi:N^\vee\ra\Theta_{\pi_{\mbf{p}}},
  \theta\mapsto \xi_\theta$. To see that the image of $\xi$ freely generates
  $\Theta_{\pi_{\mbf{p}}}$ over $\qq[T\oplus S]/I$ take $\theta_1,\dots,\theta_n$
  generators of $N^\vee$. Then $\xi_{\theta_1},\dots,\xi_{\theta_n}$ are independent over
  $\qq[T\oplus S]/I$ and, in view of the commutative diagram (\ref{eq:23}), we see that
  they generate the module $\Theta_{\mathcal{X}_{\psi(q)}}$ of derivations of the
  coordinate ring of the $\psi(q)$-fibre of $\pi_{\mbf{p}}$. As no derivation can be
  supported only at the $0$-fibre of $\pi_{\mbf{p}}$, we obtain the result.
\end{proof}
\begin{rem}
  \label{prop:4} C.~Sabbah points out that the ring $\qq[N]^\Sigma$ can also be described
  as the graded algebra associated to the Newton filtration induced by $\beta$ on $N_\qq$.
  More precisely, let $\mathcal{P}$ be the convex hull of $\beta(\mathbf{e}_i)$ in
  $N_\qq$. It is a convex polyhedron containing the origin whose faces are defined by
  $L_i=1$, being $L_i$ the unique $\qq$-linear on $N_\qq$ with
  $L_i(\beta(\mathbf{e}_j))=1$ for $i\not=j$. It thus defines the fan $\Sigma$. For $m\in
  N$, let us set $\nu(m):=\min\set{\lambda\geq 0:m\in\lambda\cdot\mathcal{P}}$ and define
  $\qq[N]_\nu$ as the vector space generated by $\mathbf{y}^{\mathbf{m}}$ with
  $\nu(\mathbf{m})\leq\nu$. It is readily seen that the convexity of $\mathcal{P}$ implies
  that $\nu(\mathbf{m}+\mathbf{m}')\leq\nu(\mathbf{m})+\nu(\mathbf{m}')$ with equality if
  and only there exists a cone $\sigma\in\Sigma$ containing both $\mathbf{m}$ and
  $\mathbf{m'}$. Hence we have
  \begin{equation}
    \label{eq:25} \qq[N]^\Sigma\simeq
    \text{gr}_{\mathcal{P}}\qq[N]=\bigoplus_{\nu\geq
      0}\frac{\qq[N]_\nu}{\qq[N]_{<\nu}}
  \end{equation} In fact if we set
  $\abs{\mathbf{b}}=\sum_{i=0}^n\frac{a_i}{w_i}$ for
  $\mathbf{b}=(b_0,\dots,b_n)\in\qq^{n+1}$, it is easy to see that
  \begin{equation}
    \label{eq:26}
    \nu(\mathbf{m}+\mathbf{m}')=\abs{\alpha(\mathbf{m})}+\abs{\alpha(\mathbf{m}')}
    -\abs{\gamma\big(\alpha(\mathbf{m})+\alpha(\mathbf{m}')\big)},
  \end{equation} and in particular,
  $\nu(\mathbf{m})=\abs{\alpha(\mathbf{m})}$. This formula can be used
  to identify the fibration constructed in theorem~\ref{prop:3} with a
  certain noetherian subring of $\oplus_{\nu\geq 0}\qq[N]_\nu$.
\end{rem}
\begin{rem}
  \label{prop:5} It is reasonable to expect some relation between the fibration
  (\ref{eq:23}) and the small quantum cohomology of weighted projective spaces as
  described in \cite{CCLTqocwps}.
\end{rem}

\section{About the well-formed condition}

As we explain in the introduction after Theorem \ref{thm:intro}, we have not used that the
weights $\mbf{p}$ are well--formed. The proposition below justify this well--formed
assumption.

\begin{prop}\label{prop:iso,mirror,fibration}
  Let $(\mbf{w},\mbf{p})$ and $(\mbf{w'},\mbf{p'})$ two pairs of weights such that
  $\gcd(\mbf{p})=\gcd(\mbf{p}')=1$. If the toric orbifolds $\mathcal{X}(\mbf{w},\mbf{p})$
  and $\mathcal{X}(\mbf{w'},\mbf{p'})$ are isomorphic then there exists two isomorphisms
  $g$ and $h$ such that the following diagram is commutative:
  \begin{equation}
    \xymatrix{
      \mathcal{Y}(\mbf{p})\ar[rr]^{g} \ar[dd]^{\pi_{\mbf{p}}}
      \ar[rd]^{f_{\mbf{w}}}&&\mathcal{Y}(\mbf{p'})\ar[dl]_{f_{\mbf{w'}}}\ar[dd]^{\pi_{\mbf{p'}}}\\ 
      & \qq& \\ 
      \mathcal{C}(\mbf{p})\ar[rr]^{h}&&\mathcal{C}(\mbf{p'})}
  \end{equation}
\end{prop}

We start with a combinatorial Lemma.
\begin{lem}\label{lem:comb}
  For any $i\in\set{0, \ldots ,n-1}$ we have that for any $(j,k)\in\set{0, \ldots
    ,n-1}\times \nn$
  \begin{equation}
    \frac{1}{p_{i}}\overline{kp_{j}}^{p_{i}}=
    \frac{1}{ap_{i}}\overline{kap_{j}}^{ap_{i}}
  \end{equation}
\end{lem}

\begin{proof}Without loss of generality, we can assume that $i=0$. There exists unique
  $(q,\overline{kp_{j}}^{p_{0}})\in\nn\times \set{0, \ldots ,p_{0}-1}$ and unique
  $(q',\overline{kap_{j}}^{p_{0}})\in\nn\times \set{0, \ldots ,ap_{0}-1}$ such that $
  kp_{j}=qp_{0}+\overline{kp_{j}}^{p_{0}}$ and $
  akp_{j}=q'ap_{0}+\overline{kap_{j}}^{ap_{0}}$.  By uniqueness, we deduce that $aq'=q$
  and $a.\overline{kp_{j}}^{p_{0}}=\overline{kap_{j}}^{ap_{0}}$. This finishes the proof.
\end{proof}

\begin{proof}[Proof of Proposition \ref{prop:iso,mirror,fibration}] According to the
  discussion at the end of Section \ref{sec:toric-stacks-polyt} and Remark
  \ref{rem:iso,w,p}, it is enough to prove the proposition for the weights
  $(\mbf{w},ap_{0}, \ldots ,ap_{n-1},p_{n})$ and $(w_{0}, \ldots ,w_{n-1},aw_{n},\mbf{p})$
  where $\gcd(p_{0}, \ldots ,p_{n})=1$ and $\gcd(a,p_{n})=1$.

  We will see that the isomorphism $g:\mathcal{Y}(ap_{0}, \ldots ,ap_{n-1},p_{n}) \to
  \mathcal{Y}(p_{0}, \ldots ,p_{n})$ sends $(z_{0}, \ldots ,z_{n})$ to $(z_{0}, \ldots
  ,z_{n-1},z_{n}^{1/a})$ and $h:\mathcal{C}(ap_0,\ldots,ap_{n-1},p_n)\to
  \mathcal{C}(p_0,\ldots,p_n)$ sends $t$ to $t^{1/a}$.

  In our notation, we will stress for which family of weights we compute $S,T,\ldots$. We
  define a morphism of semigroups $h^{\ast}: T(p_0,\ldots,p_n)\to
  T(ap_0,\ldots,ap_{n-1},p_n)$ that sends $\gamma$ to ${\gamma}/{a}$.  For any
  $i,j\in\{0,\ldots,n-1\}$, we have $\lcm(ap_i,ap_j)=a\lcm(p_i,p_j)$ and for any
  $i\in\{0,\ldots,n\}$, we have $\lcm(ap_i,p_n)=a\lcm(p_i,p_n)$.  We deduce that
  $h:\mathcal{C}(ap_0,\ldots,ap_{n-1},p_n)\to \mathcal{C}(p_0,\ldots,p_n)$ is
  well--defined and is an isomorphism.

  We define the morphism of semigroups :
  \begin{align}
    \phi^{\ast} :S(p_0,\ldots,p_n) &\to S(ap_0,\ldots,ap_{n-1},p_n) \\
    (b_0,\ldots,b_n)&\mapsto \left(b_0,\ldots,b_{n-1},\frac{b_n}{a}\right).\nonumber
  \end{align}
  We will show that $\phi^{\ast}$ is an isomorphism.  To prove that $\phi^{\ast}$ is
  well--defined, we show that:
  \begin{enumerate}
  \item\label{item:2} for $i\in\{0,\ldots,n\}$, we have
    $\phi^{\ast}(S_i(p_0,\ldots,p_n))\subset S_i(ap_0,\ldots,ap_{n-1},p_n)$
  \item\label{item:3} and for $i\in\set{0,\ldots,n}$, we have $\phi^{\ast}(\mbf{e}_i)\in
    S(ap_0,\ldots,ap_{n-1},p_n)$.
  \end{enumerate}
  (a). For $i=n$, it is obvious. For the case $i\neq n$, one can assume, without loss of
  generality, that $i=0$. Let $ \frac{1}{p_{0}}(0,\overline{kp_{1}}^{p_{0}}, \ldots
  ,\overline{kp_{n}}^{p_{0}})$ be a generator of $S_0(p_0,\ldots,p_n)$.  Let $u,v\in \nn$
  such that $up_n-va=1$. As we have that for any $t\in\nn$
  \begin{equation}
    (k+tup_0)p_n=\left(\left\lfloor\frac{kp_n}{ap_0}\right\rfloor+v\right)ap_0+\overline{kp_{n}}^{ap_{0}}+tp_0,
  \end{equation}
  we deduce that from there exists $t\in \nn$ such that
  $0\leq\overline{(k+tup_0)p_n}^{ap_0}\leq p_0 -1$.  Putting $k':=k+tup_{0}$, we have that
  $\overline{k'p_n}^{ap_0}=\overline{kp_n}^{p_0}$ and by Lemma \ref{lem:comb} that for
  $j\in\{0, \ldots ,n-1\}$ $\overline{ak'p_{j}}^{ap_{0}}=\overline{akp_{j}}^{ap_{0}}$.
  We deduce that $\phi^{\ast}(S_0(p_0,\ldots,p_n))\subset S_0(ap_0,\ldots,ap_{n-1},p_n)$.

  \vskip\baselineskip
  \noindent{}(b). For $i\in\set{1,\ldots,n-1}$, we have $\phi^{\ast}(\mbf{e}_i)\in
  S(ap_0,\ldots,ap_{n-1},p_n)$. For $i=n$, we put $k=up_0$ where $u,v\in \nn$ are the
  Bezout coefficients (i.e.  $up_n-va=1$), we deduce that
  \begin{equation}
    \phi^{\ast}(\mbf{e}_n)=\frac{1}{ap_0}\left(0,\overline{kap_1}^{ap_0},\ldots,\overline{kap_{n-1}}^{ap_0},\overline{kp_n}^{ap_0}\right).
  \end{equation}
  We conclude that $\phi^\ast$ is well--defined.

  The morhism $\phi^{\ast}$ is clearly injective and its image is contained in
  $S(ap_0,\ldots,ap_{n-1},p_n)$. It hence suffices to show that it is also surjective. It
  is obvious that $S_n(ap_0,\ldots,ap_{n-1},p_n)\subset \im \phi^\ast$.  Let
  $\frac{1}{ap_{0}}(0,\overline{kap_{1}}^{ap_{0}}, \ldots
  ,\overline{kap_{n-1}}^{ap_{0}},\overline{kp_{n}}^{ap_{0}})$ be a generator of
  $S_0(ap_0,\ldots,ap_{n-1},p_n)$ and consider the Euclidean division $
  \overline{kp_n}^{ap_0}=qp_0+\overline{kp_n}^{p_0}$ with $0\leq\overline{kp_n}^{p_0}<
  p_0$, by Lemma \ref{lem:comb}. We deduce that
  \begin{equation*}
    \phi^\ast\left(\frac{1}{p_0}\left(0,\overline{kp_1}^{p_0},\ldots,\overline{kp_n}^{p_0}\right)
      +(0,\ldots,0,q)\right)=\frac{1}{ap_{0}}\left(0,\overline{kap_{1}}^{ap_{0}}, \ldots
      ,\overline{kap_{n-1}}^{ap_{0}},\overline{kp_{n}}^{ap_{0}}\right).
  \end{equation*}
  By the same argument, we deduce that $S_j(ap_0,\ldots,ap_{n-1},p_n)\subset \im
  \phi^\ast$.  We conclude that $\phi^\ast$ is an isomorphism.  We define the isomorphism
  of rings $\widetilde{g}^\ast:=(h^\ast,\phi^\ast):\qq[(T\oplus S)(p_0,\ldots,p_n)] \to
  \qq[(T\oplus S)(ap_0,\ldots,ap_{n-1}, p_n)]$.

  As for any $b\in S$ we have $h^{\ast}(\gamma(b))=\gamma(\phi^{\ast}(b))$ and $
  \phi^{\ast}(\alpha(b))=\alpha(\phi^{\ast}(b))$, we deduce that $\widetilde{g}^\ast$
  induces an isomorphism of rings
  \begin{equation}
    g^\ast:  \qq[(T\oplus
    S)(p_0,\ldots,p_n)]/I(p_0,\ldots,p_n) \to \qq[(T\oplus
    S)(p_0,\ldots,p_n)]/I(ap_0,\ldots,ap_{n-1},p_n)
  \end{equation}
  The induced morphism of schemes $g: \mathcal{Y}(ap_{0}, \ldots ,ap_{n-1},p_{n}) \to
  \mathcal{Y}(p_{0}, \ldots ,p_{n})$ sends $(z_{0}, \ldots ,z_{n})$ to $(z_{0}, \ldots
  ,z_{n-1},z_{n}^{1/a})$ is an isomorphism such that the diagram of Proposition
  \ref{prop:iso,mirror,fibration} is commutative.

\end{proof}
\section{Examples}
\label{sec:chow-rings-as}

In this section we have used {\sc Singular} {\cite{Singular})} to embed the fibration
\eqref{eq:23} into affine spaces. We illustrated for two different cases:
$\mathbf{p}=(2,3,5)$ and $\mathbf{p}=(p_0,\dots,p_n)$ where $p_0=1$ and $p_i$ divides
$p_{i+1}$. We begin with the latter.\newline

\noindent{\bf Case $\pp(\mathbf{p})$ with $p_i|p_{i+1}$.} Set $p_{i}=d_ip_{i-1}$. Let
$v_i=z_0^{p_0/p_i}\dots z_{i-1}^{p_{i-1}/p_i}$ for $i=1,\dots,n$. The monomials $v_i$
correspond to the generators of $S$ described in \ref{prop:2}, (a). We thus have the
relations $v_{i+1}^{d_{i+1}}-v_{i}z_{i}$, $i=1,\dots,n-1$ and it is easy to see that in
fact these generate all the relations between the monomials in $v_i$ and $z_i$. On the
other hand, we have $\qq[T]=\qq[t^{1/p_n}]\simeq\qq[s]$ and the ideal $I\in\qq[T\oplus S]$
is generated by one single element, namely $v_nz_n-s=0$. Then
\begin{equation}
  \label{eq:27} \mathcal{X}
  =\set{v_1^{d_1}=z_0,v_2^{d_2}=v_1z_1,v_3^{d_3}=v_{2}z_2,\dots,v_n^{d_n}=v_{n-1}z_{n-1},v_nz_n=s}
  \hookrightarrow \aa^{2n+2}
\end{equation} and $\pi:\mathcal{X}\ra\aa^1$ is the restriction of the
projection onto the $s$-line.  \\

\noindent{\bf Case $\pp(2,3,5)$.} To describe $\qq[S]$ for the case $\mathbf{p}=(2,3,5)$,
consider the monomials
\begin{equation}
  \label{eq:28} w_1=z_0^{2/5}z_1^{3/5}, w_2=z_0^{4/5}z_1^{1/5},
  w_3=z_0^{1/5}z_1^{4/5},w_4=z_0^{3/5}z_1^{2/5}
\end{equation} with relations
\begin{equation}
  \label{eq:29}
  \begin{array}{c} w_1^2=w_2z_1,w_1w_2=w_3z_0,w_1w_3=w_4z_1,
    w_1w_4=z_0z_1\\ w_2
  \end{array}
\end{equation} Similarly, we have monomials $u_1$ and $v_1$
corresponding to the elements with denominator $1/2$ and $1/3$ with
relations
\begin{equation}
  \label{eq:30} u_1^2=z_1z_2,v_1^3=z_0z_3.
\end{equation} On the other hand we have
$\qq[T]=\qq[t^{1/6},t^{1/10},t^{1/15}]$ so that we have the embedding
$\mathcal{C}\hookrightarrow\aa^3$ as the rational curve
\begin{equation}
  \label{eq:31} s_1=s_2s_3,s_2^3=s_1s_3^2,s_3^3=s_2^2
\end{equation} Finally, the ideal $I$ is given by
\begin{equation}
  \label{eq:32}
  \begin{array}{c} u_1v_1=s_1\\
    u_1w_1=s_2w_2,u_1w_2=s_2w_4,u_1w_3=s_2z_1,u_1w_4=s_2w_1\\
    v_1w_1=s_3w_4,v_1w_2=s_3z_0,v_1w_3=s_3w_1,v_1w_4=s_3w_2
  \end{array}
\end{equation} Therefore $\mathcal{X}\hookrightarrow\aa^{12}$ is
defined by the equations \eqref{eq:29}, \eqref{eq:30}, \eqref{eq:31}
and \eqref{eq:32}, with the fibration $\pi:\mathcal{X}\ra\mathcal{C}$
given by the projection onto the $(s_1,s_2,s_3)$-space.

Notice that in any of the above cases we can obtain presentations of
$A_\orb^*(\mathcal{X}(\mathbf{w},\mathbf{p}))$ by setting $s=0$ and adding the equations
\begin{equation}
  \label{eq:33}
  \frac{w_i}{p_i}z_i^{w_i}-\frac{w_{i+1}}{p_{i+1}}z_i^{w_{i+1}},
  i=0,\dots,n-1.
\end{equation}

\begin{ack}
  The first author is grateful to B.~Dubrovin and B.~Fantechi for making it possible to
  visit the Scuola Internazionale Superiore di Studi Avanzati where this article was
  started.
\end{ack}

\providecommand{\bysame}{\leavevmode\hbox to3em{\hrulefill}\thinspace}
\providecommand{\MR}{\relax\ifhmode\unskip\space\fi MR }
\providecommand{\MRhref}[2]{%
  \href{http://www.ams.org/mathscinet-getitem?mr=#1}{#2}
}
\providecommand{\href}[2]{#2}

{\it
  \begin{tabular}{cc}
    \parbox{.4\textwidth}{
      Ignacio de Gregorio \\
      Mathematics Institute \\
      University of Warwick \\
      CV4 7AL, Coventry \\
      United Kingdom \\
      Tel.: +44 2476 524 402 \\
      {\tt I.de-Gregorio@warwick.ac.uk}\\}
    &
    \parbox{.5\textwidth}{
      {\'E}tienne Mann\\
      Institut de Math{\'e}matiques et Mod{\'e}lisation\\
      Universit{\'e} Montpellier 2\\
      Case Courrier 051\\
      Place Eug{\'e}ne Bataillon\\
      34095 Montpellier Cedex\\
      France \\
      {\tt emann@math.univ-montp2.fr}}
  \end{tabular}
}


\begin{thebibliography}{CCLT06}

\bibitem[AGV02]{AGVaoqc} Dan Abramovich, Tom Graber, and Angelo Vistoli, \emph{Algebraic
    orbifold quantum products}, Orbifolds in mathematics and physics (Madison, WI, 2001),
  Contemp. Math., Amer. Math. Soc., Providence, RI, 2002, pp.~1--24.

\bibitem[Bar00]{Bms} Serguei Barannikov, \emph{{Semi-infinite Hodge structures and mirror
      symmetry for projective spaces}}, arXiv.Math.AG/0010157 (2000), 17.

\bibitem[BCS05]{BCSocdms05} Lev~A. Borisov, Linda Chen, and Gregory~G. Smith, \emph{{The
      orbifold Chow ring of toric Deligne-Mumford stacks}}, J. Amer. Math. Soc. (2005),
  no.~1, 193--215 (electronic).

\bibitem[CCLT06]{CCLTqocwps} Tom Coates, Alessio Corti, Yuan-Pin Lee, and Hsian-Hua Tseng,
  \emph{The quantum orbifold cohomology of weighted projective space}, math.AG/0608481
  (2006).

\bibitem[CR04]{CRnco} Weimin Chen and Yongbin Ruan, \emph{A new cohomology theory of
    orbifold}, Comm.  Math. Phys. (2004), no.~1, 1--31.

\bibitem[dG06]{IDGdd-06} Ignacio de~Gregorio, \emph{Deformations of functions and
    {$F$}-manifolds}, Bull. London Math. Soc. \textbf{38} (2006), no.~6, 966--978.

\bibitem[FMN07]{FMNts-07} Barbara Fantechi, Etienne Mann, and Fabio Nironi, \emph{Smooth
    toric dm stacks}, arXiv.AG:0708.1254 (2007).

\bibitem[Giv98]{Gmttci98} Alexander Givental, \emph{A mirror theorem for toric complete
    intersections}, Topological field theory, primitive forms and related topics (Kyoto,
  1996), Progr. Math., vol. 160, Birkh\"auser Boston, Boston, MA, 1998, pp.~141--175.

\bibitem[GPS05]{Singular} G.-M. Greuel, G. Pfister, and H. Sch\"onemann.  {\sc Singular}
  3.0. A Computer Algebra System for Polynomial Computations. Centre for Computer Algebra,
  University of Kaiserslautern (2005). {\tt http://www.singular.uni-kl.de}.


\bibitem[Man06]{mann-2006} Etienne Mann, \emph{Orbifold quantum cohomology of weighted
    projective spaces}, arXiv.math.AG/0610965 (2006).




\end{thebibliography}
\end{document}